\documentclass{amsart}
\usepackage{shortsalch, xy, mathabx}
\xyoption{all}
\title{How many adjunctions give rise to the same monad?}
\date{}
\begin{document}
\begin{abstract}
Given an adjoint pair of functors $F,G$, the composite $GF$ naturally gets the structure of a monad. The same monad may arise from many such adjoint pairs of functors, however. Can one describe all of the adjunctions giving rise to a given monad?
In this paper we single out a class of adjunctions with especially good properties, and we develop methods for computing all such adjunctions, up to natural equivalence, which give rise to a given monad. To demonstrate these methods, we explicitly compute the finitary homological presentations of the free $A$-module monad on the category of sets, for $A$ a Dedekind domain. We also prove a criterion, reminiscent of Beck's monadicity theorem, for when there is essentially (in a precise sense) only a single adjunction that gives rise to a given monad.
\end{abstract}

\maketitle
\tableofcontents

\section{Introduction.}

If $\mathcal{C},\mathcal{D}$ are categories and $G: \mathcal{D}\rightarrow \mathcal{C}$ a functor with a left adjoint $F$, then 
the composite $GF$ gets the structure of a 
monad. However, given a monad $T: \mathcal{C}\rightarrow \mathcal{C}$, there may be {\em many} categories $\mathcal{D}$ and 
adjoint pairs $F,G$ such that
$GF = T$ as a monad. We will call such a choice of $\mathcal{D},F,$ and $G$ a {\em presentation} for $T$. 

It has long been known that, 
among all presentations of a given monad $T$, there is an {\em initial} presentation, the Kleisli category of $T$, and a {\em terminal} presentation, the 
Eilenberg-Moore category of $T$. Furthermore, Beck's monadicity theorem gives a necessary and sufficient condition on $G$ for
the presentation $(\mathcal{D},F,G)$ to be equivalent to the Eilenberg-Moore category. (See \cite{MR1712872} for a nice exposition of these ideas.) 
Beck's result has proven very useful, e.g. in algebraic geometry where, in its dual form for comonads, it is the foundation for the general theory of descent; see 
\cite{MR0255631}.

The applications of Beck's monadicity theorem have made it very clear that, given a presentation $(\mathcal{D},F,G)$ of a monad $T$, it is very useful
to be able to tell when $(\mathcal{D},F,G)$ is the terminal presentation (i.e., the Eilenberg-Moore category) of $T$. However, we have long wondered
about what to do when one encounters a presentation of a monad which is {\em not} the terminal (Eilenberg-Moore) presentation, and not even the initial (Kleisli)
presentation. Can one describe the collection of all presentations of a monad? Even better, can one establish some kind of {\em coordinate system}
on the collection of presentations of a monad, so that when one encounters a presentation of a monad (or a comonad, e.g. for applications in descent theory), 
one can give some kind of ``coordinates'' that describe where this
presentation sits in relation to the initial and terminal presentations, and all the other presentations, of the same monad?

In this paper we study the collection of presentations of a given monad, but with a restriction on what presentations we are willing to consider. This is because, given a presentation $(\mathcal{D},F,G)$, one can
trivially produce many more presentations by taking the Cartesian product of $\mathcal{D}$ with any small category. We regard these presentations as
degenerate, and we want to disregard presentations with this kind of redundant information in them. Consequently, in Definition~\ref{def of homological},
we make the definition that a presentation $(\mathcal{D},F,G)$ is said to be {\em homological} if every object $X$ of $\mathcal{D}$ can be 
recovered from the $F,G$-bar construction on $X$ (see Definition~\ref{def of homological} for the precise definition). This definition eliminates the ``redundant''
presentations we wanted to exclude, and has some other good properties, described in Remark~\ref{remark on def of homological}. We also restrict our attention
to what we call ``coequalizable'' monads, that is, those monads $T$ for which the Eilenberg-Moore category has coequalizers; this property is satisfied in
all cases of interest which we know of, and in Remark~\ref{remark on coequalizability} we explain a bit about why that is.

Once these definitions are made, we can prove some nice theorems:
\begin{itemize}
\item In Theorem~\ref{main thm}, we prove that, if $T$ is coequalizable, then the category of natural equivalence classes of homological presentations of $T$ is
equivalent to the partially-ordered collection of reflective replete subcategories of the Eilenberg-Moore category $\mathcal{C}^T$ which contain the Kleisli category $\mathcal{C}_T$.

This means the category of all natural equivalence classes of homological presentations of $T$ is always well-behaved in at least one way: it can't be just any arbitrary category, rather, it is always partially-ordered (i.e., there is at most one morphism from any given object to any other given object). 
\item In Theorem~\ref{coordinatization thm}, when $T$ is coequalizable and $\mathcal{C}^T$ has a biproduct and is Krull-Schmidt,
we actually construct a ``coordinate system'' on the natural equivalence classes of homological presentations of $T$! 
Any homological presentation is determined uniquely (up to natural equivalence) by specifying a subcollection of the collection of isomorphism classes of indecomposable
objects of $\mathcal{C}^T$. In some practical cases, $\mathcal{C}^T$ is the 
category of finitely generated modules over an algebra, and then the vertices of the Auslander-Reiten quiver of $\mathcal{C}^T$ act as ``coordinates'' for the collection of natural equivalence classes of homological presentations of $T$; see Remark~\ref{coordinatization remark}.
\item In Theorem~\ref{criterion for uhp}, we give a simple and usable criterion for the {\em triviality} (up to natural equivalence) of the collection of homological presentations of $T$, i.e.,
a criterion for when, up to natural equivalence, there exists only one homological presentation of $T$ (necessarily the Eilenberg-Moore category of $T$). 
As an example, in Corollary~\ref{field extension monads are uniquely homologically presentable} we show that the base-change monad on module categories
associated to a field extension has this property of unique homological presentability.
\item In Theorem~\ref{main dd corollary}, we 
use Theorems~\ref{main thm} and Theorem~\ref{coordinatization thm}
to explicitly compute the collection of all (natural equivalence classes of) homological finitary presentations of each monad in one particular class of monads: namely, let $A$ be a Dedekind domain, and let $T$ denote the monad on the category of sets which sends a set to the underlying set of the free $A$-module it generates.
We show that the partially-ordered set of natural equivalence classes of finitary homological presentations of $T$ is equivalent to the 
set of functions
\[ \Max\Spec(A) \rightarrow \{ 0, 1, 2, \dots , \infty\} \]
from the set $\Max\Spec(A)$ of maximal ideals of $A$ to the set of extended
natural numbers, under the partial ordering in which we let
$f\leq g$ if and only if $f(\mathfrak{m}) \leq g(\mathfrak{m})$ for all $\mathfrak{m}\in\Max\Spec(A)$. Since the ring of integers $\mathbb{Z}$ is a Dedekind domain, this is a very fundamental example.
(Here a presentation for $T$ is ``finitary'' if its right adjoint functor preserves filtered colimits; this is a condition that, roughly speaking, guarantees that the data of the presentation is determined by ``finite input.'' See Remark~\ref{remark on finitariness}.)
\end{itemize}

In many cases of interest (e.g. the base-change monads on module categories associated to maps of rings or maps of schemes),
the Eilenberg-Moore category $\mathcal{C}^T$ is actually {\em abelian}, hence $T$ is coequalizable and $\mathcal{C}^T$ has a biproduct automatically,
and frequently $\mathcal{C}^T$ is actually quite computable and understandable. Under those circumstances our results seem to be fairly useful, and as we
hope Theorem~\ref{main dd corollary} demonstrates, they are actually applicable and give explicit nontrivial results in concrete situations of interest.

We are very grateful to the anonymous referee for perspicacious comments which helped us to improve this paper greatly.

\section{Homological presentations of a monad are equivalent to replete reflective subcategories of its Eilenberg-Moore category.}

\subsection{Preliminary definitions.}

Throughout this paper, when $T$ is a monad, when convenient we will 
sometimes also write $T$ for the underlying functor of the monad.

\begin{definition}
Let $\mathcal{C}$ be a category, $T$ a monad on $\mathcal{C}$. If $\mathcal{D}$ is a category equipped with a functor $F^{\prime}: \mathcal{C}\rightarrow \mathcal{D}$
and a right adjoint $G^{\prime}$ for $F^{\prime}$ such that the associated monad $G^{\prime}F^{\prime}$ is equal to $T$, we call the data $(\mathcal{D},F^{\prime},G^{\prime})$ a {\em presentation of $T$.}
Sometimes we shall just write $\mathcal{D}$ as shorthand for $(\mathcal{D},F^{\prime},G^{\prime})$, when $F^{\prime},G^{\prime}$ are clear from context.

The collection of all presentations of $T$ forms a large category, whose morphisms are morphisms of adjunctions (see IV.7 of \cite{MR1712872} for the definition
of morphisms of adjunctions) which are the identity on $\mathcal{C}$. We call this large category {\em the category of presentations of $T$}, and for which we will write $\Pres(T)$.
One also can consider natural transformations between morphisms of adjunctions, and we regard two morphisms as {\em homotopic} if there exists an invertible-up-to-isomorphism natural transformation between them; 
we will write $\Ho(\Pres(T))$ for the category of presentations of $T$ but whose morphisms are homotopy classes of morphisms of adjunctions.
\end{definition}
We note that $\Pres(T)$ and $\Ho(\Pres(T))$ are not necessarily categories, but are {\em large categories}, because their hom-collections are not necessarily hom-sets. 
The notion of a category
of presentations of a monad appears in VI.5 of \cite{MR1712872}, but was not there given a name.

For the purposes of this paper we will mostly be studying presentations of a monad
which have the property that there is some minimal degree of compatibility
between the category and the monad, 
enough compatibility to guarantee that e.g. some constructions
in homological algebra can be made. Here is our definition:
\begin{definition}\label{def of homological}
Let $\mathcal{C}$ be a category, $T$ a monad, $(\mathcal{D},F^{\prime},G^{\prime})$ a presentation of $T$. We will say that $\mathcal{D}$ is a {\em homological presentation of $T$}
if, for every $T$-algebra $\rho: G^{\prime} F^{\prime} X \rightarrow X$, 
the coequalizer (in $\mathcal{D})$ of the two natural counit maps  
\begin{equation}\label{counit maps} \epsilon_{F^{\prime}X}, F^{\prime}\rho: F^{\prime}G^{\prime}F^{\prime}X\rightarrow F^{\prime}X \end{equation}
exists, and for each object $Y$ of $\mathcal{D}$, the canonical map
\begin{equation}\label{coeq condition} \coeq\{ \epsilon_{F^{\prime}G^{\prime}Y}, F^{\prime}G^{\prime}\epsilon_Y\} \rightarrow Y\end{equation}
is an isomorphism. 
We will write $\HPres(T)$ for the full large subcategory of $\Pres(T)$ generated by the 
homological presentations.

We will write $\Ho(\HPres(T))$ for the 
large category whose objects are homological presentations of $T$, and whose morphisms are homotopy classes of morphisms of adjunctions.
\end{definition}

\begin{remark}\label{remark on def of homological}
The reason for the name ``homological'' for this kind of presentation is the following: if $(\mathcal{D},F^{\prime},G^{\prime})$ is a homological presentation of a monad and 
$\mathcal{D}$ is abelian, then each object $X$ in $\mathcal{D}$ admits a canonical resolution
\begin{equation}\label{bar res} 0 \leftarrow X \leftarrow F^{\prime}G^{\prime}X \stackrel{d_0}{\longleftarrow} F^{\prime}G^{\prime}\ker (\epsilon_{F^{\prime}G^{\prime}X}-F^{\prime}G^{\prime}\epsilon_X) \stackrel{d_1}{\longleftarrow} \dots \end{equation}
obtained by repeatedly applying $F^{\prime}G^{\prime}$, forming the coequalizer \ref{coeq condition}, and taking the kernel of the coequalizer map.
This resolution gives us a way to compute the left-derived functors of any functor on $\mathcal{D}$ which is acyclic on every object of the form $F^{\prime}G^{\prime}X$.
If $(\mathcal{D},F^{\prime},G^{\prime})$ fails to be homological, then at least for some objects $X$ the chain complex~\ref{bar res} fails to be exact and hence
cannot be used to compute derived functors in this way.

The resolution~\ref{bar res} is very familiar and commonplace in its various special cases. 
For example, when $\mathcal{C}$ is the category of sets and $T$ the monad given on a set $S$ 
by taking the underlying set of the free abelian group generated by $S$, then the category $\Ab$ is a presentation for $T$, and it is homological 
(because it is the terminal presentation, i.e., the Eilenberg-Moore category of $T$, which in Corollary~\ref{cor on e-m adjunction being homological} 
we prove is always homological for
any coequalizable monad $T$). The resolution~\ref{bar res} is the elementary resolution one uses in a first course in homological algebra to prove that free resolutions exist
in the category of abelian groups: given an abelian group $X$, one can form the direct sum $\oplus_{x\in X} \mathbb{Z}$, one can let $X_0$ be the kernel of the
obvious surjection $\oplus_{x\in X} \mathbb{Z}\rightarrow X$, then iterate to form a free resolution of $X$. 

When $\mathcal{D}$ fails to be abelian, instead of~\ref{bar res} one forms the simplicial resolution
\[ \xymatrix{ \dots \ar[r] & F^{\prime}G^{\prime}F^{\prime}G^{\prime}F^{\prime}G^{\prime}X \ar[r] & F^{\prime}G^{\prime}F^{\prime}G^{\prime}X \ar[r] & F^{\prime}G^{\prime}X }\]
of $X$, and one can use this resolution to compute more general kinds of derived functors (e.g. if $\mathcal{D}$ has the structure of a model category). 
In every case the condition that $(\mathcal{D},F^{\prime},G^{\prime})$ is homological is really the condition that $F^{\prime},G^{\prime}$ gives us a way to form a canonical resolution 
of any object in $\mathcal{D}$. In some sense one should think of a homological presentation for a monad $T$ as {\em a category equipped with a way of forming
a canonical resolution of any object by $T$-free objects,} and that means that this paper is in some sense really about classifying various ways of forming
canonical resolutions.

Finally, one more note about the map~\ref{coeq condition}: after applying $G^{\prime}$, the map always becomes an isomorphism, because the cofork
\[\xymatrix{ 
G^{\prime}F^{\prime}G^{\prime}F^{\prime}G^{\prime}X \ar@<1ex>[r]^{G^{\prime}\epsilon_{F^{\prime}G^{\prime}X}} \ar@<-1ex>[r]_{G^{\prime}F^{\prime}G^{\prime}\epsilon_X} & G^{\prime}F^{\prime}G^{\prime}X \ar[r] & G^{\prime}X
}\]
is always split by the unit map $\eta_{G^{\prime}X}: G^{\prime}X\rightarrow G^{\prime}F^{\prime}G^{\prime}X$, hence the cofork is a split coequalizer. But the map~\ref{coeq condition} can fail to
be an isomorphism {\em before} applying $G^{\prime}$.
\end{remark}

Informally, the general trend is that monads tend to either have a single homological presentation, up to equivalence, or a truly enormous collection of homological presentations, big enough to make it very difficult to explicitly classify them. However, even within the homological presentations, there is an even more restricted class of presentations of a monad which we can reasonably restrict our attention to, namely, the {\em finitary} homological presentations:
\begin{definition}
Let $\mathcal{C}$ be a category, $T$ a monad, $(\mathcal{D},F^{\prime},G^{\prime})$ a presentation of $T$. We will say that $\mathcal{D}$ is a {\em finitary} presentation of $T$ if $G^{\prime}$ preserves all filtered colimits which exist in $\mathcal{D}$.

We will write $\Fin\Pres(T)$ for the large category of finitary presentations of $T$, $\Fin\HPres(T)$ for the large category of finitary homological presentations, and $\Ho(\Fin\HPres(T))$ for the large category of finitary homological presentations up to natural equivalence.
\end{definition}
\begin{remark}\label{remark on finitariness}
In the category of modules over a ring, every object is a filtered colimit of finitely generated modules; consequently, when $(\mathcal{D},F^{\prime},G^{\prime})$ is a finitary presentation and $\mathcal{D}$ a category of modules over a ring, then $G^{\prime}$ can be computed on any object if one knows how to compute $G^{\prime}$ on finitely generated modules. This is actually quite useful; see the proof of Theorem~\ref{main dd corollary}, for example.
\end{remark}

\begin{remark}
It is {\em not} in general true that every presentation of a monad $T$ is finitary, even if $T$ itself preserves filtered colimits; for example, let $\mathcal{C}$ be the category of sets, let $T$ be the monad which sends a set $S$ to the underlying set of the free abelian group generated by $S$, and let $(\mathcal{D},F^{\prime},G^{\prime})$ be the presentation for $T$ in which $\mathcal{D}$ is the category of {\em reduced} abelian groups, $F^{\prime}: \mathcal{C}\rightarrow\mathcal{D}$ is the free abelian group functor, and $G^{\prime}: \mathcal{D} \rightarrow \mathcal{D}$ is the forgetful functor. Then, for any prime number $p$, 
the colimit of the filtered diagram
\[ \mathbb{Z} \stackrel{p}{\longrightarrow} \mathbb{Z} \stackrel{p}{\longrightarrow} \dots \]
in $\mathcal{D}$ is zero, but the colimit of 
\[ G^{\prime}(\mathbb{Z}) \stackrel{G^{\prime}(p)}{\longrightarrow} G^{\prime}(\mathbb{Z}) \stackrel{G^{\prime}(p)}{\longrightarrow} \dots \]
in $\mathcal{C}$ is the underlying set of $\mathbb{Z}[\frac{1}{p}]$.
\end{remark}

Recall that a subcategory is said to be {\em replete} if it contains every object isomorphic to one of its own objects,
and {\em reflective} if it a full subcategory and the inclusion of the subcategory admits a left adjoint.
We include fullness as part of the definition of a reflective subcategory; this seems to be relatively standard, although not universal
(see e.g.~\cite{MR1712872}), in the literature.

\begin{definition}
Let $\mathcal{C}$ be a category, $T$ a monad, $\mathcal{C}^T$ the Eilenberg-Moore category of $T$-algebras. If $\mathcal{D}$ is a replete reflective subcategory
of $\mathcal{C}^T$, we will say that $\mathcal{D}$ {\em presents $T$} if $\mathcal{D}$ contains all the free $T$-algebras, i.e., if $\mathcal{D}$
contains the $T$-algebra $TTX\stackrel{\mu_X}{\longrightarrow} TX$ for every object $X$ of $\mathcal{C}$.

We will write $\Loc(T)$ for the 
partially-ordered collection of all replete reflective subcategories $\mathcal{C}^T$ which present $T$.

Finally, we will say that an element $\mathcal{D}$ of $\Loc(T)$ is {\em finitary} if the forgetful functor from $\mathcal{D}$ to $\mathcal{C}^T$ preserves all filtered colimits which exist in $\mathcal{D}$. We will write $\Fin\Loc(T)$ for the subcollection of $\Loc(T)$ consisting of the finitary elements.
\end{definition}
The notation ``$\Loc(T)$'' is motivated by the notion that a replete reflective subcategory of a category is, speaking intuitively and roughly, a kind of ``localization'' of that category.

We note that $\Loc(T)$ is not necessarily a set, {\em nor even a class} (we are grateful to Mike Shulman for pointing out to us that the collection of
subcategories of a category is not necessarily a class!). Sometimes the term ``conglomerate'' is used for a collection too large to form a class. 
In other words, if one wants to use Grothendieck universes, one must expand the universe {\em twice} to go from sets to conglomerates.
In practical algebraic, geometric, and topological situations, however, it seems likely that $\Loc(T)$ will form a set. For example, see Corollary~\ref{cor on hpres being a poset}, where we show that mild conditions (a biproduct condition, a Krull-Schmidt condition, and a smallness condition) imply that $\Loc(T)$ is a set.

Finally, it will sometimes be convenient to have coequalizers in Eilenberg-Moore categories. We introduce a definition which describes
monads which have this agreeable property:
\begin{definition}\label{def of coequalizable}
Let $\mathcal{C}$ be a category, $T$ a monad, $\mathcal{C}^T$ the Eilenberg-Moore category of $T$-algebras. 
We will say that $T$ is {\em coequalizable} if $\mathcal{C}^T$ has coequalizers.
\end{definition}

\begin{remark}\label{remark on coequalizability}
There are many known conditions on $T$ which guarantee that $T$ is coequalizable; for example, in Lemma II.6.6 in \cite{MR1417719} it is shown
that, if $T$ preserves reflexive coequalizers, then $T$ is coequalizable. Consequently, many interesting examples of monads $T$ are coequalizable.

For example, suppose $R \rightarrow S$ is a map of commutative rings. Then the base-change monad $T: \Mod(R) \rightarrow \Mod(R)$, i.e., the composite of 
the extension of scalars functor $\Mod(R)\rightarrow \Mod(S)$ with the restriction of scalars functor $\Mod(S)\rightarrow\Mod(R)$,
is coequalizable, since extension of scalars and restriction of scalars are both right exact, preserving all coequalizers. If $S$ is finitely generated
as an $R$-module then the base-change monad $\fgMod(R)\rightarrow\fgMod(R)$ on the finitely generated module category is also coequalizable, for the same
reason. Then the Eilenberg-Moore category $\Mod(R)^T$ is equivalent to $\Mod(S)$.

More generally, if $f: Y\rightarrow X$ is a map of schemes and $\QC\Mod(\mathcal{O}_X)$ the category of quasicoherent $\mathcal{O}_X$-modules,
then the base-change monad $f_*f^*$ is coequalizable if $f$ is an affine morphism, since in that case $f_*$ is right exact (and $f^*$ is always right exact, regardless
of whether $f$ is affine).
Then the Eilenberg-Moore category $\QC\Mod(\mathcal{O}_X)^{f_*f^*}$ is equivalent to $\QC\Mod(\mathcal{O}_Y)$, by the results of EGA II.1.4, \cite{MR0163909}.

Usually (e.g. in the examples above, and in
our Theorem~\ref{main dd corollary}) we will have an explicit description of the category $\mathcal{C}^T$ and we will
know that it has coequalizers; what will be interesting and new will be the description of the {\em rest} of $\HPres(T)$.
\end{remark}

\begin{definition}\label{def of weakly k-s}
If $\mathcal{C}$ is a category with coproduct $\oplus$, we say that an object $X$ of $\mathcal{C}$ is {\em indecomposable} if $X\cong Y\oplus Z$
implies either $Y\cong 0$ or $Z\cong 0$. We say that $\mathcal{C}$ is {\em weakly Krull-Schmidt} if every object $X$ of $\mathcal{C}$ admits a decomposition
into a finite coproduct of indecomposable objects, and that decomposition is unique up to permutation of the summands.
\end{definition}
Definition~\ref{def of weakly k-s} differs from the usual definition of a Krull-Schmidt category in that we do not require the indecomposable objects to have local endomorphism rings.

\subsection{Replete reflective subcategories presenting a monad are equivalent to its homological presentations.}

\begin{lemma}\label{main lemma 1}
Let $\mathcal{C}$ be a category, $T$ a monad on $\mathcal{C}$, $(\mathcal{D},F^{\prime},G^{\prime})$ 
a presentation of $T$. If $\mathcal{D}$ has coequalizers of all pairs of 
maps of the form~\ref{counit maps}, then 
the canonical comparison functor
$K: \mathcal{D}\rightarrow \mathcal{C}^T$ has a left adjoint.
Conversely, if $T$ is coequalizable and $K$ is full and faithful and has a left adjoint, then $\mathcal{D}$ has all coequalizers (and in particular, all pairs of maps
of the form~\ref{counit maps}).
\end{lemma}
\begin{proof}
We will write $F: \mathcal{C}\rightarrow\mathcal{C}^T$ for the canonical functor and $G$ for its right adjoint.
When $\mathcal{D}$ has coequalizers of all parallel pairs of the form~\ref{counit maps}, the comparison functor $K$ admits a left adjoint $V$, defined
on objects as follows: if $TX\stackrel{\rho_X}{\longrightarrow} X$ is the structure map of a $T$-algebra, then $V$ applied to that $T$-algebra is
the coequalizer of the maps 
\[ F^{\prime} \rho_X,\epsilon_{F^{\prime}}X: F^{\prime}G^{\prime}F^{\prime} X \rightarrow F^{\prime}X,\]
using the fact that $G^{\prime}F^{\prime} = GF = T$.
(The result that $K$ has a left adjoint if $\mathcal{D}$ has coequalizers is an old one: it appears in Beck's thesis \cite{beckthesis}, and even appears as an exercise in VI.7 of \cite{MR1712872}. But the only coequalizers one actually needs
are the ones used in the construction of the left adjoint, i.e., those of 
the form~\ref{counit maps}.)

For the converse: suppose $T$ is coequalizable and $K$ is full and faithful and has a left adjoint $V$. Since left adjoints preserve colimits and since fullness and faithfulness of $K$ is equivalent to $VK\cong \id_{\mathcal{D}}$, we can compute
the coequalizer of any pair $f,g: X\rightarrow Y$ in $\mathcal{D}$ by computing the coequalizer of $Kf,Kg$ in $\mathcal{C}^T$ (which exists since $T$ is coequalizable)
and then applying $V$.
\end{proof}

\begin{lemma}\label{main lemma 2}
Let $\mathcal{C}$ be a category, $T$ a monad on $\mathcal{C}$, $(\mathcal{D},F^{\prime},G^{\prime})$ 
a presentation of $T$. Suppose $\mathcal{D}$ has coequalizers of all pairs of maps
of the form~\ref{counit maps}. 
Then the comparison functor $K: \mathcal{D}\rightarrow \mathcal{C}^T$ is full and faithful if and only if 
$(\mathcal{D},F^{\prime},G^{\prime})$ is homological.
\end{lemma}
\begin{proof}
We use the same notation as in the proof of Lemma~\ref{main lemma 1}.
That $K$ is full and faithful is equivalent to the counit map $VK \rightarrow \id_{\mathcal{D}}$ of the adjunction being an isomorphism.
We recall that $K$ is defined on objects by letting $KX$ be the $T$-algebra with structure map $G^{\prime}F^{\prime}G^{\prime}X = TG^{\prime}X \rightarrow G^{\prime}X$
given by the counit natural transformation $F^{\prime}G^{\prime}\rightarrow \id_{\mathcal{D}}$. Now $VKX$ is precisely the coequalizer of the two maps
\[ \epsilon_{F^{\prime}G^{\prime} X}, F^{\prime}G^{\prime}\epsilon_X: F^{\prime}G^{\prime}F^{\prime}G^{\prime} X\rightarrow F^{\prime}G^{\prime} X,\]
and the map $VKX\rightarrow X$ is precisely the map~\ref{coeq condition}. So the condition that $(\mathcal{D},F^{\prime},G^{\prime})$ be homological
is equivalent to the condition that $VK\rightarrow\id_{\mathcal{D}}$ be an isomorphism of functors, i.e., the condition that $K$ be full and faithful.
\end{proof}

\begin{theorem}\label{main thm}
Let $\mathcal{C}$ be a category, $T$ a coequalizable 
monad on $\mathcal{C}$. Then the large homotopy category $\Ho(\HPres(T))$ of homological presentations of $T$ is equivalent to the
partially-ordered collection $\Loc(T)$ of replete reflective subcategories of $\mathcal{C}^T$ which present $T$. 

Furthermore, if the forgetful functor $\mathcal{C}^T \rightarrow \mathcal{C}$ preserves filtered colimits,
then this equivalence restricts to an equivalence between 
the subcollection $\Ho(\Fin\HPres(T))$ of $\Ho(\HPres(T))$
and the subcollection $\Fin\Loc(T)$ of $\Loc(T)$.
\end{theorem}
\begin{proof}
We write $\Kl(T)$ for the Kleisli category of $T$, we write $F:\mathcal{C}\rightarrow\mathcal{C}^T$ for the canonical functor and $G$ for its right adjoint, and we write
$F^{\prime\prime}: \mathcal{C}\rightarrow \Kl(T)$ for the canonical functor
and $G^{\prime\prime}$ for its right adjoint. 
The theorem 
follows almost immediately from Lemma~\ref{main lemma 2}, which gives us that every homological presentation $(\mathcal{D},F^{\prime},G^{\prime})$ of $T$ has the
property that $K$ is faithful and full, hence $\mathcal{D}$ is canonically equivalent to a full replete subcategory of $\mathcal{C}^T$,
and Lemma~\ref{main lemma 1}, which gives us that that full replete subcategory is reflective.
That reflective replete subcategory contains the free $T$-algebras, i.e.,
the Kleisli category of $T$, since the Kleisli category is initial among
presentations of $T$. So that reflective replete subcategory of 
$\mathcal{C}^T$ is an element of $\Loc(T)$. If we furthermore assume that $G$ preserves filtered colimits and
$(\mathcal{D},F^{\prime},G^{\prime})$ is finitary, then since $G^{\prime} = G\circ K$, for any filtered diagram $\mathcal{X}$ in $\mathcal{D}$,
we have the natural commutative diagram
\[\xymatrix{
 \colim G^{\prime}(\mathcal{X}) \ar[rr]^{\cong} \ar[rd]^{\cong} & & G^{\prime}(\colim \mathcal{X}) \\
 & G(\colim K(\mathcal{X})), \ar[ru] &
}\]
so the map $G(\colim K(\mathcal{X})) \rightarrow G^{\prime}(\colim \mathcal{X}) = G(K(\colim \mathcal{X}))$ is an isomorphism, and since $G$ reflects isomorphisms,
$K$ preserves filtered colimits.

Conversely, if $\mathcal{D}$ is a reflective replete subcategory of $\mathcal{C}^T$ with inclusion
$K: \mathcal{D}\rightarrow\mathcal{C}^T$ having left adjoint $V$, suppose we write
$S: \Kl(T)\rightarrow \mathcal{D}$ for the inclusion of the free $T$-algebras.
We claim that the composite $S\circ F^{\prime\prime}: \mathcal{C}\rightarrow
\mathcal{D}$ has right adjoint $G\circ K: \mathcal{D}\rightarrow\mathcal{C}$,
and that the composite monad $G\circ K \circ S \circ F^{\prime\prime}$ 
is equal to the monad $T$. The second claim is very easy: the composite
$G \circ K \circ S$ is equal to $G^{\prime\prime}$, so 
\[ G\circ K \circ S \circ F^{\prime\prime} = 
G^{\prime\prime}\circ F^{\prime\prime} = T.\]
The first claim is also not difficult: since $F = K\circ S \circ F^{\prime\prime}$ and $V\circ K \simeq \id_{\mathcal{D}}$,
we have 
\[ V\circ F \simeq V\circ K\circ S\circ F^{\prime\prime} \simeq S \circ F^{\prime\prime}.\]
Now $V$ is left adjoint to $K$ and $F$ left adjoint to $G$, so 
$S\circ F^{\prime\prime}\simeq V\circ F$ is left adjoint to $G\circ K$,
proving our first claim. It follows that $(\mathcal{D},S\circ F^{\prime\prime},
G\circ K)$ is a presentation of $T$.

All that remains to be proven is that $(\mathcal{D},S\circ F^{\prime\prime}, G\circ K)$ is a {\em homological} presentation of $T$. 
By construction,  $K$ is full and faithful,
so by Lemma~\ref{main lemma 1}, $\mathcal{D}$ has coequalizers of all parallel pairs of the form~\ref{counit maps}.
So by Lemma~\ref{main lemma 2}, $(\mathcal{D},S\circ F^{\prime\prime}, G\circ K)$ is homological.
If we furthermore assume that $G$ preserves filtered colimits and that $\mathcal{D}$ is finitary, then of course the composite $G\circ K$ preserves filtered colimits, and consequently $(\mathcal{D},S\circ F^{\prime\prime}, G\circ K)$ is also a finitary homological presentation.
\end{proof}

\begin{corollary} \label{cor on e-m adjunction being homological}
If $T$ is coequalizable, the Eilenberg-Moore adjunction $(\mathcal{C}^T, F, G)$ of $T$ is a homological presentation of $T$. (And, consequently, the terminal homological presentation of $T$.)
\end{corollary}

\begin{corollary}
If $T$ is coequalizable, the large homotopy category $\Ho(\HPres(T))$ of homological presentations of $T$ is partially-ordered, i.e., for 
any objects $\mathcal{A},\mathcal{B}$ of $\Ho(\HPres(T))$, there is at most one morphism $\mathcal{A}\rightarrow\mathcal{B}$.
If we furthermore assume that the forgetful functor $\mathcal{C}^T \rightarrow \mathcal{C}$ preserves filtered colimits, then large homotopy category $\Ho(\Fin\HPres(T))$ of finitary homological presentations of $T$ is also partially-ordered.
\end{corollary}

\begin{corollary}\label{cor on hpres being a poset}
Suppose $T$ is coequalizable and $\mathcal{C}^T$ is weakly Krull-Schmidt.
Suppose the collection of 
isomorphism classes of indecomposable objects forms a {\em set} (not a proper class!), and suppose that set has cardinality $\kappa$.
Then $\Ho(\HPres(T))$ is equivalent to a partially-ordered {\em set} of cardinality no greater than $2^{\aleph_{0}^{\kappa}}$.
\end{corollary}
\begin{proof}
The partially-ordered collection $\Loc(T)$, which by Theorem~\ref{main thm} is equivalent to \linebreak $\Ho(\HPres(T))$,
is contained in the collection of subcollections of the collection of finite formal sums of indecomposable objects.
This collection is, in turn, contained in the collection of subcollections of the collection of not-necessarily-finite formal sums of indecomposable objects in which
each indecomposable object appears only finitely many times. This last collection has cardinality $2^{\aleph_{0}^{\kappa}}$.
\end{proof}
We greatly improve this cardinality bound in Corollary~\ref{better cardinality bound} under the assumption that $\mathcal{C}^T$ has a biproduct.

\subsection{Coordinatization of the collection of homological presentations of a monad.}

\begin{definition}
Recall that a category $\mathcal{C}$ is said to {\em have a biproduct} if it has finite products and finite coproducts and, for each finite family $\{ X_i\}_{i\in I}$
of objects of $\mathcal{C}$, the canonical map $\coprod_{i\in I} X_i \rightarrow \prod_{i\in I} X_i$ is an isomorphism.
\end{definition}

Lemmas~\ref{localizations have biproducts} and~\ref{lemma on sums and summands} are easy and must be well-known, but we do not know where they already appear in the literature.
\begin{lemma}\label{localizations have biproducts}
Suppose $\mathcal{A}$ is a replete reflective subcategory of a category with biproduct. Then $\mathcal{A}$ has a biproduct.
\end{lemma}
\begin{proof}
Let $\mathcal{C}$ be a category with biproduct, let $\mathcal{A}$ be a replete reflective subcategory of $\mathcal{C}$, let $G: \mathcal{A}\rightarrow\mathcal{C}$
denote the inclusion functor, and let $F: \mathcal{C}\rightarrow\mathcal{A}$ denote its left adjoint. Given objects $X,Y$ of $\mathcal{C}$, 
we claim that $F(G(X) \oplus G(Y))$ is a biproduct for $X$ and $Y$ in $\mathcal{A}$. Since $F$ is a left adjoint, it is trivial that $F(G(X)\oplus G(Y))$ is a coproduct for $X$ and $Y$ in $\mathcal{A}$.
To see that $F(G(X)\oplus G(Y))$ is also a product for $X$ and $Y$ in $\mathcal{A}$,
suppose $T$ is an object of $\mathcal{A}$, and observe that we have natural bijections
\begin{align*}
 \hom_{\mathcal{A}}(T, X) \times \hom_{\mathcal{A}}(T, Y) 
  & \cong  \hom_{\mathcal{C}}(G(T), G(X) \prod G(Y)) \\
  & \cong  \hom_{\mathcal{C}}(G(T), G(X) \oplus G(Y)) \\
  & \cong  \hom_{\mathcal{C}}(G(T), GF(G(X) \oplus G(Y))) \\
  & \cong  \hom_{\mathcal{A}}(T, F(G(X) \oplus G(Y))),\end{align*}
hence $F(G(X)\oplus G(Y))$ has the universal property of the product in $\mathcal{A}$.
\end{proof}

\begin{lemma}\label{lemma on sums and summands}
Suppose $\mathcal{A}$ is a replete reflective subcategory of a weakly Krull-Schmidt category $\mathcal{C}$ with biproduct $\oplus$. If
$X\cong Y \oplus Z$ in $\mathcal{C}$, then $X$ is in $\mathcal{A}$ if and only if both $Y$ and $Z$ are in $\mathcal{A}$.
\end{lemma}
\begin{proof}
We write $L: \mathcal{C}\rightarrow \mathcal{C}$ for the composite of the reflector functor $\mathcal{C}\rightarrow\mathcal{A}$ with the inclusion 
$\mathcal{A}\rightarrow\mathcal{C}$. 
By Lemma~\ref{localizations have biproducts}, $\mathcal{A}$ has a biproduct.
Since $L$ is a composite of a left adjoint (the reflector functor) with a 
right adjoint (the inclusion functor), it preserves biproducts, since the biproduct is both the finite coproduct and the finite product.
So $LX\cong LY\oplus LZ$, and if $Y,Z$ are in $\mathcal{A}$, then the unit maps $Y\rightarrow LY$ and $Z\rightarrow LZ$ are both isomorphisms.
So $X \cong Y\oplus Z \rightarrow LY\oplus LZ \cong LX$ is an isomorphism. So $X$ is in $\mathcal{A}$.

For the converse, suppose $X$ is in $\mathcal{A}$. Let $X\cong \oplus_{i=1}^n X_i$ be the decomposition of $X$ into indecomposables, given by the weak Krull-Schmidt 
property. Then the unit map 
\[\oplus_{i=1}^n X_i\cong X\rightarrow LX\cong \oplus_{i=1}^n LX_i\]
is an isomorphism, and it is the sum of the component maps $X_i\rightarrow LX_i$, so we have some permutation
\[ \sigma: \{ 1, \dots ,n\} \rightarrow \{ 1,\dots n\}\]
with the property that $LX_i \cong X_{\sigma(i)}$. However, since $\mathcal{A}$ is a replete reflective subcategory, $L$ is {\em idempotent}, so
$\sigma\circ \sigma = \sigma$. So $\sigma$ must be the identity permutation. So each component map $X_i\rightarrow LX_i$ is an isomorphism.
Hence if $X$ splits as a direct sum and $X$ is in $\mathcal{A}$, each summand is also in $\mathcal{A}$.
\end{proof}

\begin{theorem} {(\bf Coordinatization.)} \label{coordinatization thm} Let $\mathcal{C}$ be a category, $T$ a coequalizable 
monad on $\mathcal{C}$. Suppose the Eilenberg-Moore category $\mathcal{C}^T$ has a biproduct and is weakly Krull-Schmidt. 
Write $\Gamma(\mathcal{C}^T)$ for the collection of isomorphism classes of indecomposable objects in $\mathcal{C}^T$.
Then $\Ho(\HPres(T))$ embeds by an order-preserving map into the collection of subcollections of $\Gamma(\mathcal{C}^T)$.
\end{theorem}
\begin{proof}
By Theorem~\ref{main thm}, specifying an element of $\Ho(\HPres(T))$ is equivalent to specifying a replete reflective subcategory of $\mathcal{C}^T$ which contains $\mathcal{C}_T$,
hence is determined uniquely by which isomorphism classes of objects in $\mathcal{C}^T$ are
contained in the replete reflective subcategory. But by Lemma~\ref{lemma on sums and summands}, a replete reflective subcategory of a weakly Krull-Schmidt category with biproduct
is determined uniquely by which indecomposables are contained in it.
\end{proof}
In other words: under the conditions of Theorem~\ref{coordinatization thm}, a homological presentation of $T$
can be specified by specifying a subcollection of $\Gamma(\mathcal{C}^T)$ 
(which, as we describe in the last section of this paper, is actually computable in cases of interest). 
Since $\mathcal{C}^T$ is often computable and understandable, Theorem~\ref{coordinatization thm}---when
it applies---gives a coordinatization of the collection of homological presentations of $T$, as desired.

\begin{corollary} \label{better cardinality bound} 
Suppose $T$ is coequalizable and $\mathcal{C}^T$ is weakly Krull-Schmidt and has a biproduct.
Suppose the collection of 
isomorphism classes of indecomposable objects forms a {\em set} (not a proper class!), and suppose that set has cardinality $\kappa$.
Then $\Ho(\HPres(T))$ is equivalent to a partially-ordered {\em set} of cardinality no greater than $2^{\kappa}$.
\end{corollary}

\begin{remark}\label{coordinatization remark}
Suppose that $k$ is a field and $A$ is a $k$-algebra.
Let $\mathcal{C}$ be any category and $T$ any monad on $\mathcal{C}$ whose Eilenberg-Moore category $\mathcal{C}^T$ is equivalent to the category $\fgMod(A)$ of finitely generated $A$-modules; for example, we could let $\mathcal{C}$ be the category of $B$-modules, for some reasonable subalgebra $B$ of $A$, and we could let $T$ be the monad associated to the free-forgetful adjunction between $B$-modules and $A$-modules.
In Auslander-Reiten theory, the set of isomorphism classes of indecomposable finitely generated $A$-modules is exactly the set $\Gamma(\fgMod(A))$ 
of vertices in the Auslander-Reiten quiver of $\mathcal{C}^T$. 
So one can regard the vertices of the Auslander-Reiten quiver of $\mathcal{C}^T$ as ``coordinates'' for the collection of natural equivalence classes of homological presentations of $T$: Theorem~\ref{coordinatization thm} gives 
an embedding of $\Ho(\Pres(T))$ into the partially-ordered set of subsets of $\Gamma(\fgMod(A))$.
\end{remark}

\section{A criterion for unique homological presentability of a monad.}

\subsection{Preliminary definitions.}

Some monads can (up to natural equivalence) only be homologically presented in a single way, i.e., $\Ho(\HPres(T))$ is equivalent to a one-object category.  Here is the relevant definition:
\begin{definition}
Suppose $T$ is a monad. If $\Ho(\HPres(T))$ has only a single element, then we say that $T$ is {\em uniquely homologically presentable.}
\end{definition}
We give a concrete algebraic class of examples (base-change monads associated to field extensions) of uniquely homologically presentable monads
in Corollary~\ref{field extension monads are uniquely homologically presentable}. 

Because we will need to make use of it, we state Beck's monadicity theorem (see e.g. VI.7 of \cite{MR1712872}):
\begin{theorem} {\bf (Beck.)} \label{beck's thm}
Suppose $\mathcal{C},\mathcal{D}$ are categories, $G: \mathcal{D}\rightarrow\mathcal{C}$ a functor with a left adjoint $F$.
Then the comparison functor $\mathcal{D}\rightarrow\mathcal{C}^{GF}$ is an equivalence of categories if and only if,
whenever a parallel pair $f,g:X\rightarrow Y$ in $\mathcal{D}$ is such that $Gf,Gg$ has a split coequalizer in $\mathcal{C}$,
each of the following conditions hold:
\begin{itemize}
\item $f,g$ has a coequalizer $\coeq\{f,g\}$ in $\mathcal{D}$,
\item $G$ preserves the coequalizer of $f,g$, i.e., the natural map \linebreak $\coeq\{ Gf,Gg\}\rightarrow G\coeq\{ f,g\}$ is an isomorphism,
\item and $G$ reflects the coequalizer of $f,g$, i.e., if $Z$ is a cocone over the diagram $f,g:X\rightarrow Y$ such that $GZ$ is a coequalizer of 
$Gf,Gg$, then $Z$ is a coequalizer of $f,g$.
\end{itemize}
\end{theorem}
Here is a very classical definition:
\begin{definition} 
When $G$ is a functor with left adjoint, we say that $G$ is {\em monadic} if $G$ satisfies the equivalent conditions of Theorem~\ref{beck's thm}.
\end{definition}
We offer a (to our knowledge, new) variant on this definition which will be essential to our criterion for unique homological presentability of a monad.
\begin{definition}
Suppose $\mathcal{C},\mathcal{D}$ are categories, $G: \mathcal{D}\rightarrow\mathcal{C}$ a functor.
We say that $G$ is {\em absolutely monadic} if $G$ has a left adjoint and, whenever a parallel pair $f,g:X\rightarrow Y$ in $\mathcal{D}$ is such that $Gf,Gg$ has a split coequalizer in $\mathcal{C}$,
then:
\begin{itemize}
\item $f,g$ has a {\em split} coequalizer $\coeq\{f,g\}$ in $\mathcal{D}$,
\item $G$ preserves the coequalizer of $f,g$, i.e., the natural map \linebreak $\coeq\{ Gf,Gg\}\rightarrow G\coeq\{ f,g\}$ is an isomorphism,
\item and $G$ reflects the coequalizer of $f,g$, i.e., if $Z$ is a cocone over the diagram $f,g:X\rightarrow Y$ such that $GZ$ is a coequalizer of 
$Gf,Gg$, then $Z$ is a coequalizer of $f,g$.
\end{itemize}
\end{definition}
Note that a functor that is absolutely monadic is also monadic, but the converse does not always hold.

\subsection{A criterion for unique homological presentability.} Now we present and prove the main result of this section.

First we will need a lemma. We suspect that this lemma is already well-known, but we do not know an already-existing reference in the literature.
\begin{lemma}\label{reflectors are beck}
Suppose $\mathcal{D},\mathcal{E}$ are categories, and $\mathcal{D}\stackrel{S}{\longrightarrow} \mathcal{E}$
is a full, faithful functor with a left adjoint. Then $S$ is monadic.
\end{lemma}
\begin{proof}
Since $S$ is full and faithful, we regard it as inclusion of a subcategory $\mathcal{D}$ of $\mathcal{E}$. 
Then since $S$ has a left adjoint, $\mathcal{D}$ is a reflective subcategory of $\mathcal{E}$.
We write $V$ for the left adjoint of $S$.
Let $f,g: X\rightarrow Y$ be a pair of maps in $\mathcal{D}$ such that $Sf,Sg$ has a split coequalizer $Z$
Then we can apply $V$ together with the natural equivalence $VS\simeq \id_{\mathcal{D}}$ 
to get that $VZ$ is a split coequalizer of $f,g$. Hence $S$ sends a cofork in $\mathcal{D}$ to a split coequalizer in $\mathcal{E}$ if and only if the cofork
was already a split coequalizer in $\mathcal{D}$. So $S$ preserves coequalizers of all pairs in $\mathcal{D}$ with a $S$-split coequalizer, and since $S$ is
faithful and injective on objects, it reflects isomorphisms; so $S$ is monadic.
\end{proof}

\begin{theorem}\label{criterion for uhp}
Suppose $\mathcal{C}$ is a category, $T$ a coequalizable monad on $\mathcal{C}$. We write $F$ for the canonical functor $\mathcal{C}\rightarrow\mathcal{C}^T$
and $G$ for its right adjoint. If $G$ is absolutely monadic, then $T$ is uniquely homologically presentable.
\end{theorem}
\begin{proof}
Suppose $G$ is absolutely monadic, and suppose that $(\mathcal{D},F^{\prime},G^{\prime})$ is a presentation of $T$. We have the comparison functor
$K: \mathcal{D}\rightarrow \mathcal{C}^T$, and we have that $G^{\prime} = G\circ K$. 
We are going to show that, if $(\mathcal{D},F^{\prime},G^{\prime})$ is homological, then $K$ is an equivalence.

Suppose $f,g$ is a parallel pair in 
$\mathcal{D}$ such that $G^{\prime}f,G^{\prime}g$ has a split coequalizer in $\mathcal{C}$. Then $Kf,Kg$ is a parallel pair in $\mathcal{C}^T$ such that
$G(Kf),G(Kg)$ has a split coequalizer in $\mathcal{C}$, and since $G$ is absolutely monadic, $Kf,Kg$ has a split coequalizer $Z$ such that $GZ$ is the given split
coequalizer for $G^{\prime}f,G^{\prime} g$. But, by Lemma~\ref{reflectors are beck}, $K$ is monadic, hence, by Theorem~\ref{beck's thm}, $f,g$ has a
coequalizer $W$ such that $KW$ is $Z$. Hence $f,g$ has a coequalizer in $\mathcal{D}$ and $G^{\prime}$ preserves that coequalizer.

Now we check that $G^{\prime}$ reflects appropriate coequalizers. Suppose 
\begin{equation}\label{starting cofork} \xymatrix{ X \ar@<1ex>[r]^{f} \ar@<-1ex>[r]_{g} & Y \ar[r] & Z } \end{equation} 
is a cofork in $\mathcal{D}$ such that the cofork
\begin{equation} \xymatrix{ G^{\prime}X \ar@<1ex>[r]^{G^{\prime}f} \ar@<-1ex>[r]_{G^{\prime}g} & G^{\prime}Y \ar[r] & G^{\prime}Z } \end{equation} 
is a split coequalizer sequence in $\mathcal{C}$. Again using the fact that $G^{\prime} = G\circ K$
and using that $G$ is absolutely monadic, we have that the cofork
\begin{equation} \xymatrix{ KX \ar@<1ex>[r]^{Kf} \ar@<-1ex>[r]_{Kg} & KY \ar[r] & KZ } \end{equation} 
in $\mathcal{C}^T$ is a split coequalizer sequence; finally, by Lemma~\ref{reflectors are beck} and Theorem~\ref{beck's thm},
$K$ reflects such coequalizers, so cofork~\ref{starting cofork} is a coequalizer sequence in $\mathcal{D}$.

Hence $G^{\prime}$ preserves and reflects coequalizers of all parallel pairs $f,g$ such that $Gf,Gg$ has a split coequalizer. Hence,
by Theorem~\ref{beck's thm}, $G^{\prime}$ is monadic, and the comparison map $\mathcal{D}\rightarrow \mathcal{C}^{G^{\prime}F^{\prime}} = \mathcal{C}^T$ is an equivalence of categories. 

Hence every homological presentation $\mathcal{D}$ of $T$ is equivalent to the entire Eilenberg-Moore category of $T$. Hence $\Ho(\HPres(T))$ consists of a
single element, the Eilenberg-Moore presentation.
\end{proof}

\begin{corollary}
Suppose $\mathcal{C}$ is an abelian category and $T$ a monad on $\mathcal{C}$ such that $\mathcal{C}^T$ is abelian and the canonical
functor $G: \mathcal{C}^T\rightarrow\mathcal{C}$ is additive. Suppose that, if 
\[  X \rightarrow Y \rightarrow Z \]
is a pair of maps in $\mathcal{C}^T$ such that
\[ GX \rightarrow GY \rightarrow GZ \rightarrow 0\]
is split exact in $\mathcal{C}$, then
\[ X \rightarrow Y \rightarrow Z \rightarrow 0 \]
is split exact in $\mathcal{C}^T$.
Then $T$ is uniquely homologically presentable.
\end{corollary}
\begin{proof}
The assumed condition on $G$ is precisely what absolute monadicity of $G$ means in the abelian setting.
\end{proof}

\begin{corollary}\label{field extension monads are uniquely homologically presentable}
Suppose $L/K$ is a field extension and $T: \Mod(K) \rightarrow \Mod(K)$ the associated base change monad, i.e., $TM$ is the underlying $K$-module of $L\otimes_K M$. 
Then $T$ is uniquely homologically presentable.
\end{corollary}

\section{Explicit examples: Dedekind domains.}

First, recall the well-known classification of finitely generated modules over a Dedekind domain, which we will use throughout this section:
\begin{theorem}\label{classification of modules over dedekind domain}
Let $A$ be a Dedekind domain, and let $M$ be a finitely generated $A$-module. Then $M$ is isomorphic to a direct sum of a finitely generated projective $A$-module and finitely many $A$-modules of the form $A/\mathfrak{m}^n$ for various maximal ideals $\mathfrak{m}$ in $A$ and various positive integers $n$.
\end{theorem}

Now here is a result which we will use in the proof of Theorem~\ref{main dd corollary}:
\begin{prop}\label{explicit dedekind domain example}
Let $A$ be a Dedekind domain and let $\fgMod(A)$ denote the category of finitely generated $A$-modules. Then the partially-ordered set of reflective replete subcategories of $\fgMod(A)$ which contain the free $A$-modules is
is isomorphic to the
set of functions
\[ \Max\Spec(A) \rightarrow \{ 0, 1, 2, \dots , \infty\} \]
from the set $\Max\Spec(A)$ of maximal ideals of $A$ to the set of extended
natural numbers, under the partial ordering in which we let
$f\leq g$ if and only if $f(\mathfrak{m}) \leq g(\mathfrak{m})$ for all $\mathfrak{m}\in\Max\Spec(A)$.
\end{prop}
\begin{proof}
Suppose that $\mathcal{A}$ is a replete reflective subcategory of $\fgMod(A)$ which contains the free $A$-modules. By Theorem~\ref{classification of modules over dedekind domain}, $\fgMod(A)$ is weakly Krull-Schmidt, and by Lemma~\ref{lemma on sums and summands}, 
$\mathcal{A}$ is determined by which of the indecomposable finitely generated $A$-modules it contains.

Suppose that $\mathfrak{m}$ is a maximal ideal of $A$ and that $A/\mathfrak{m}^n$ is in $\mathcal{A}$, and let $i\leq n$. 
Then $A/\mathfrak{m}^i$ is the kernel of a map $A/\mathfrak{m}^n \rightarrow A/\mathfrak{m}^n$, and consequently $A/\mathfrak{m}^i$ is also in $\mathcal{A}$,
since a reflective category is closed under limits.
Consequently, $\mathcal{A}$ is completely determined by a single function
$f_{\mathcal{A}}: \Max\Spec(A) \rightarrow \{ 0, 1, 2, \dots , \infty\}$
from the set of maximal ideals of $A$ to the set of extended natural numbers; namely, $f_{\mathcal{A}}$ is the function sending a maximal ideal $\mathfrak{m}$ of $A$ to the largest integer $n$ such that $A/\mathfrak{m}^n$ is in $\mathcal{A}$, and by letting $f_{\mathcal{A}}(\mathfrak{m}) = \infty$ if $A/\mathfrak{m}^n$ is in $\mathcal{A}$ for all $n$.

We claim that, for each function 
$f: \Max\Spec(A) \rightarrow \{ 0, 1, 2, \dots , \infty\}$, there does indeed
exist a replete reflective subcategory $\mathcal{A}$ of $\fgMod(A)$ such that $f_{\mathcal{A}} = f$.
Let $\mathcal{A}_f$ denote the full subcategory of $\fgMod(A)$ generated by the $A$-modules $M$ with the property that, for each maximal ideal $\mathfrak{m}$ of $A$,
if there exists a monomorphism $A/\mathfrak{m}^n \rightarrow M$ of $A$-modules,
then $n\leq f(\mathfrak{m})$.
Clearly $\mathcal{A}_f$ contains $A$ as well as $A/\mathfrak{m}^n$ for all $n\leq f(\mathfrak{m})$,
and $\mathcal{A}_f$ does {\em not} contain $A/\mathfrak{m}^n$ if $n > f(\mathfrak{m})$.
Hence (using Theorem~\ref{classification of modules over dedekind domain}) $f_{\mathcal{A}_f} = f$.

Clearly $\mathcal{A}_f$ is full and replete in $\fgMod(A)$, so the only remaining question is whether $\mathcal{A}_f$ is reflective.
We now construct an explicit left adjoint for the inclusion functor $\mathcal{A}_f \hookrightarrow \fgMod(A)$.
Given an $A$-module $M$, let $t_f(M)$ denote the subset of $M$ consisting of all elements $x\in M$ such that,
for some maximal ideal $\mathfrak{m}$ of $A$, 
\begin{itemize} 
\item $x\in \mathfrak{m}^{f(\mathfrak{m})}M$, and 
\item there exists some positive integer $n$ such that $ax =0$ for all $a\in \mathfrak{m}^n$.
\end{itemize}
Let $u_f(M)$ denote the sub-$A$-module of $M$ generated by the subset $t_f(M)$. Clearly, if $g: M^{\prime} \rightarrow M$
is an $A$-module homomorphism, then $g(t_f(M^{\prime})) \subseteq t_f(M)$. Consequently $u_f$ is a functor from $\fgMod(A)$ to $\fgMod(A)$, and $u_f$ is equipped with a
natural transformation $i_f: u_f \rightarrow \id$, namely, the natural inclusion of $u_f(M)$ into $M$.
Let $v_f: \fgMod(A) \rightarrow \fgMod(A)$ be the functor given by $v_f(M) = \coker i_f(M)$. Clearly $v_f$ is equipped with a natural transformation $\eta_f: \id \rightarrow v_f$, namely, the natural projection of $M$ onto $M/u_f(M)$.

By Theorem~\ref{classification of modules over dedekind domain}, every finitely generated $A$-module is isomorphic to
$A^n \oplus \coprod_{i=1}^m A/\mathfrak{m}_i^{\epsilon_i}$ for some nonnegative integers $m,n$, some sequence of maximal ideals
$(\mathfrak{m}_1, \dots, \mathfrak{m}_m)$ of $A$, and some sequence of positive integers $(\epsilon_1, \dots ,\epsilon_m)$.
Clearly,
\[ v_f\left( A^n \oplus \coprod_{i=1}^m A/\mathfrak{m}_i^{\epsilon_i}\right) = A^n \oplus \coprod_{i=1}^m A/\mathfrak{m}_i^{\min\{\epsilon_i,f(\mathfrak{m})\}}\]
and $\eta_f( A^n \oplus \coprod_{i=1}^m A/\mathfrak{m}_i^{\epsilon_i})$ is the obvious projection map.
Consequently, for every finitely generated $A$-module $M$, $v_f(M)$ is in $\mathcal{A}_f$, and since every map 
from $A^n \oplus \coprod_{i=1}^m A/\mathfrak{m}_i^{\epsilon_i}$ to an object of $\mathcal{A}_f$ factors through the projection map $\eta_f$, 
the functor $v_f$, with its codomain restricted to $\mathcal{A}_f$, is left adjoint to the inclusion
$\mathcal{A}_f \hookrightarrow \fgMod(A)$. Consequently $\mathcal{A}_f$ is reflective.
\end{proof}

\begin{corollary}
Let $K,L$ be number fields, that is, finite extensions of $\mathbb{Q}$, with rings of integers $A$ and $B$, respectively. Let $L/K$ be a field extension.
Let $T$ be the monad associated to the induction-restriction adjunction between $\fgMod(A)$ and $\fgMod(B)$, i.e., $T(M)$ is the underlying $A$-module of $B\otimes_A M$.
Then $\Ho(\HPres(T))$ has only a single element. That is, there exists (up to natural equivalence) only one homological presentation of $T$.
\end{corollary}
\begin{proof}
The Kleisli category $\fgMod(B)_T$ contains all the $B$-modules of the form $B\otimes_A M$ for $M$ a finitely generated $A$-module. For any maximal ideal $\mathfrak{m}$ of $B$, and any positive integer $n$, 
$\fgMod(B)_T$ contains a module $B\otimes_A M$ with $B/\mathfrak{m}^i$ as a summand for some $i\geq n$, namely, let $\mathfrak{p}$ be the (unique) prime of $A$ under $\mathfrak{m}$, and let $M = A/\mathfrak{p}^n$.
Consequently the only reflective replete subcategory of $\fgMod(B)$ which
contains $\fgMod(B)_T$ is the one which, in the language of Proposition~\ref{explicit dedekind domain example}, corresponds to the function
$\Max\Spec(B) \rightarrow \{ 0, 1, \dots, \infty\}$ sending every maximal ideal to $\infty$, i.e., $\fgMod(B)$ itself.
\end{proof}

\begin{theorem}\label{main dd corollary}
Let $A$ be a Dedekind domain,
let $\Sets$ denote the category of sets, and
let $T$ denote the monad on $\Sets$ associated to the free-forgetful adjunction between $\Mod(A)$ and $\Sets$, i.e., $T(S)$ is the underlying set of the
free $A$-module generated by $S$. 
Then the partially-ordered collection $\Ho(\Fin\HPres(T))$ of natural equivalence classes of finitary homological presentations of $T$ is equivalent to the
set of functions
\[ \Max\Spec(A) \rightarrow \{ 0, 1, 2, \dots , \infty\} \]
from the set $\Max\Spec(A)$ of maximal ideals of $A$ to the set of extended
natural numbers, under the partial ordering in which we let
$f\leq g$ if and only if $f(\mathfrak{m}) \leq g(\mathfrak{m})$ for all $\mathfrak{m}\in\Max\Spec(A)$.
\end{theorem}
\begin{proof}
Since every $A$-module is the colimit of its finitely generated sub-$A$-modules, and since the partially-ordered set of sub-$A$-modules of a given $A$-module is filtered, we know that every $A$-module is a filtered colimit of finitely generated $A$-modules. Hence a finitary element of $\Loc(\Sets^T) \cong \Loc(\Mod(A))$
is determined by which finitely generated $A$-modules it contains. 

From here, the proof resembles that of Proposition~\ref{explicit dedekind domain example}: given a finitary element $\mathcal{A}$ of $\Loc(\Sets^T)$,
since $\mathcal{A}$ is reflective, it is closed under limits computed in $\Mod(A)$.
Consequently, if $A/\mathfrak{m}^n$ is in $\mathcal{A}$, then so is
$A/\mathfrak{m}^i$ for all $i\leq n$. Of course $A$ is also in $\mathcal{A}$.
Hence we can specify which finitely generated $A$-modules are contained in $\mathcal{A}$, and consequently all of $\mathcal{A}$, by specifying (as in the proof of Proposition~\ref{explicit dedekind domain example}) 
a function
$f_{\mathcal{A}}: \Max\Spec(A) \rightarrow \{ 0, 1, 2, \dots , \infty\}$
from the set of maximal ideals of $A$ to the set of extended natural numbers; namely, $f_{\mathcal{A}}$ is the function sending a maximal ideal $\mathfrak{m}$ of $A$ to the largest integer $n$ such that $A/\mathfrak{m}^n$ is in $\mathcal{A}$, and by letting $f_{\mathcal{A}}(\mathfrak{m}) = \infty$ if $A/\mathfrak{m}^n$ is in $\mathcal{A}$ for all $n$.

Now we still need to know that, for each 
$f: \Max\Spec(A) \rightarrow \{ 0, 1, 2, \dots , \infty\}$, there does indeed
exist a replete reflective subcategory $\mathcal{A}$ of $\Mod(A)$ such that $f_{\mathcal{A}} = f$. The argument is as follows: let $\mathcal{A}_f$ be the replete full subcategory of $\Mod(A)$ generated by all the $A$-modules $M$ with the property that,
for each maximal ideal $\mathfrak{m}$ of $A$, if $A/\mathfrak{m}^n \rightarrow M$ is a monomorphism of $A$-modules, then $n\leq f(\mathfrak{m})$.
Clearly $f_{\mathcal{A}_f} = f$ and $\mathcal{A}_f$ is replete, full, and contains the free $A$-modules. 
We need to know that $\mathcal{A}_f$ is reflective and finitary.
Let $L: \Mod(A) \rightarrow \Mod(A)$ be the functor given by letting 
$L(M)$ be the colimit $\colim_{X\in fg(M)} v_f(X)$, where
$fg(M)$ is the (filtered) category of finitely generated sub-$A$-modules of $M$, and $v_f$ is the functor defined on $\fgMod(A)$ in the proof of Proposition~\ref{explicit dedekind domain example}. Let $\eta M: M \rightarrow L(M)$ be the natural map
$M  \cong \colim_{X\in fg(M)} X \rightarrow \colim_{X\in fg(M)} v_f(X)$
given by the natural transformation $\eta_f: \id \rightarrow v_f$.
Then $L(M)$ is in $\mathcal{A}_f$, since $A/\mathfrak{m}^n$ is finitely generated
and hence every monomorphism from $A/\mathfrak{m}^n$ to the filtered colimit
$\colim_{X\in fg(M)} v_f(X)$ factors through a monomorphism $A/\mathfrak{m}^n \rightarrow v_f(X)$ for some $X\in fg(M)$.
Furthermore, if $T$ is an object of $\mathcal{A}_f$, then
\begin{align*} 
 \hom_{\Mod(A)}(M, T) 
  &\cong \lim_{X\in fg(M)} \hom_{\Mod(A)}( X, T) \\
  &\cong \lim_{X\in fg(M)} \hom_{\Mod(A)}( v_f(X), T) \\
  &\cong \hom_{\Mod(A)}( \colim_{X\in fg(M)} v_f(X), T) \\
  &\cong \hom_{\Mod(A)}( L(X), T) 
\end{align*}
so $L$ is indeed left adjoint to the inclusion $\mathcal{A}_f\hookrightarrow\Mod(A)$. (More carefully: $L = GF$, where $G$ is the inclusion $\mathcal{A}_f\hookrightarrow\Mod(A)$, and $F$ is $L$ with its codomain restricted to $\mathcal{A}_f$.)
So $\mathcal{A}_f$ is reflective.
The functor $L = GF$ commutes with filtered colimits by construction, and $F$ is a left adjoint and hence commutes with all colimits, and $G$ is full and faithful and hence reflects colimits; so $G$ commutes with filtered colimits, and hence $\mathcal{A}_f$ is finitary.
\end{proof}

\bibliography{/home/asalch/texmf/tex/salch}{}
\bibliographystyle{plain}
\end{document}